\documentclass{siamltex}
\usepackage[T1]{fontenc}
\usepackage[latin9]{inputenc}
\usepackage{float}
\usepackage{amssymb}
\usepackage{graphicx}

\makeatletter

\floatstyle{ruled}
\newfloat{algorithm}{tbp}{loa}
\providecommand{\algorithmname}{Algorithm}
\floatname{algorithm}{\protect\algorithmname}

\usepackage{algorithmic}
\usepackage{mathrsfs}

\usepackage{color}

\newcommand\eqref[1]{(\ref{#1})}

\begin{document}

\title{{\large Characterization of worst-case GMRES}}

\author{Vance Faber\footnotemark[1], J\"org Liesen\footnotemark[2] and Petr Tich\'y\footnotemark[3]}

\footnotetext[1] {Vanco Research, Big Pine Key, FL 33043 ({\tt vance.faber@gmail.com}).}

\footnotetext[2]{Institute of Mathematics, Technical University of Berlin,
Stra{\ss}e des 17. Juni 136, 10623 Berlin, Germany ({\tt liesen@math.tu-berlin.de}).
The work of this author was supported by the Heisenberg Program
of the Deutsche Forschungsgemeinschaft (DFG).}

\footnotetext[3]{Institute of Computer Science, Academy of Sciences of
the Czech Republic, Pod Vod\'arenskou v\v{e}\v{z}\'{\i} 2, 18207 Prague,
Czech Republic ({\tt tichy@cs.cas.cz}).
This work was supported by the Grant Agency of the Czech Republic under grant No. P201/13-06684~S,
and by the project M100301201 of the institutional support of the
Academy of Sciences of the Czech Republic.}

\maketitle
\begin{abstract}
Given a matrix $A$ and iteration step $k$,
we study a best possible attainable upper
bound on the GMRES residual norm that does not depend on
the initial vector $b$. This quantity is called
the worst-case GMRES approximation.
We show that the worst case
behavior of GMRES for the matrices $A$ and $A^T$ is the same, and
we analyze properties of initial vectors for which the worst-case
residual norm is attained. In particular, we show that such vectors
satisfy a certain ``cross equality'', and we characterize them as
right singular vectors of the corresponding GMRES residual matrix.
We show that the worst-case GMRES polynomial may not be uniquely determined,
and we consider the relation between the worst-case and the ideal GMRES
approximations, giving new examples in which the inequality between
the two quantities is sharp at all iteration steps $k\geq 3$.
Finally, we give a complete characterization of how the values
of the approximation problems in the context of worst-case
and ideal GMRES for a real matrix change, when one considers
complex (rather than real) polynomials and initial vectors
in these problems.
\end{abstract}

%
%   ... keywords
%

\section{Introduction}

Let a nonsingular matrix $A\in\mathbb{R}^{n\times n}$ and a vector
$b\in\mathbb{R}^{n}$ be given. Consider solving the system
of linear algebraic equations $Ax=b$ with the initial guess $x_{0}=0$ using
the GMRES method~\cite{SaSc86}. This method generates a sequence of iterates
$x_{k}\in\mathcal{K}_{k}(A,b)\equiv{\rm span}\{b,Ab,\dots A^{k-1}b\}$,
$k=1,2,\dots\,$, so that the corresponding $k$th residual $r_{k}\equiv b-Ax_{k}$
satisfies
\begin{equation}
\|r_{k}\|\;=\;\min_{p\in\pi_{k}}\,\|\, p(A)b\,\|\,.\label{eqn:GMRESAPP}
\end{equation}
Here $\|\cdot\|$ denotes the Euclidean norm, and $\pi_{k}$ denotes the set
of real polynomials of degree at most~$k$ and with value one at the origin. Note that
for a real matrix $A$ and a real right hand side $b$ the minimum in (\ref{eqn:GMRESAPP})
is achieved for a real polynomial. Considering only real polynomials therefore
does not represent any restriction.

It is clear from (\ref{eqn:GMRESAPP}), that the sequence of GMRES
residual norms $\|r_k\|$, $k=1,2,\dots\,,$ is nonincreasing.
It terminates with $r_k=0$ if and only if $k$ is equal to $d(A,b)$, the
degree of the minimal polynomial of the vector $b$ with respect to $A$.
For each $b$ we have $d(A,b)\leq d(A)$, the degree of the minimal polynomial
of $A$.

A geometric characterization of the iterate $x_{k}\in\mathcal{K}_{k}(A,b)$,
which is mathematically equivalent to (\ref{eqn:GMRESAPP}), is given by
\begin{equation}\label{eqn:GMRESORTH}
r_k \;\perp\; A\mathcal{K}_k(A,b)\,.
\end{equation}
To emphasize the dependence of the $k$th GMRES residual $r_{k}$ on the given data
$A$, $b$ and $k$ we will sometimes write
\[
r_{k}=\mathrm{GMRES}(A,b,k)\qquad\mbox{or}\qquad r_{k}=p_{k}(A)b,
\]
where $p_{k}\in\pi_k$ is the \emph{$k$th GMRES polynomial} of $A$ and $b$,
i.e., the polynomial that solves the minimization problem on the right hand side
of (\ref{eqn:GMRESAPP}). As long as $r_k\neq 0$, this polynomial is uniquely
determined. The matrix $p_{k}(A)$ is called the \emph{$k$th GMRES residual matrix}
of $A$ and~$b$. For further basic properties and algorithmic details of the GMRES
method we refer to the original paper~\cite{SaSc86} or the
books~\cite{GreBook97,LieStrBook12,SaaBook03}.

In the following we will assume without loss of generality that $\|b\|=1$.
A common approach for investigating the GMRES convergence behavior is to bound
(\ref{eqn:GMRESAPP}) independently of~$b$.
For each iteration step $k$ the best possible bound on the GMRES
residual norm that is independent of~$b$ is given by maximizing the right hand side
of (\ref{eqn:GMRESAPP}) over all unit norm vectors, i.e.,
\begin{equation}
\|r_k\|\;=\;\min_{p\in\pi_{k}}\,\|\, p(A)b\,\|\;\leq\;\max_{\|v\|=1}\min_{p\in\pi_{k}}\|\, p(A)v\,\|
\;\equiv\;\Psi_{k}(A)\,.\label{eqn:WCapp}
\end{equation}
The quantity $\Psi_k(A)$ is called \emph{the $k$th worst-case GMRES approximation}.
It is easy to see that the bound (\ref{eqn:WCapp}) is sharp in the sense that for
each given $A$ and $k$ there exists a unit norm vector $b$ so that the corresponding
$k$th GMRES residual vector satisfies $\|r_{k}\|=\Psi_{k}(A)$. We will call such a
vector $b$, the corresponding $k$th GMRES polynomial $p_{k}$ and the corresponding
$k$th GMRES residual matrix $p_{k}(A)$ the \emph{$k$th worst-case GMRES} initial
vector, polynomial and residual matrix, respectively. If $A$ is singular, then
$\Psi_k(A)=1$ for all $k\geq 0$  (to see this, simply take $b$ as a unit norm
vector in the kernel of $A$). Hence only the case of a nonsingular matrix $A$ is
of interest in this context. For such $A$ we have
$$1\geq \Psi_1(A)\geq\cdots\geq \Psi_{d(A)-1}(A)>\Psi_{d(A)}(A)=0,$$
and therefore we only need to consider $1\leq k\leq d(A)-1$.

It is known that $\Psi_{k}(A)$ for a fixed $k$ is a continuous function on the
open set of nonsingular matrices; see~\cite[Theorem~3.1]{Jo1994a}
or~\cite[Theorem~2.5]{FaJoKnMa1996}. Moreover, it was shown
in~\cite[Theorem~2.7]{FaJoKnMa1996} that $\Psi_{k}(A)=1$ for a nonsingular matrix $A$,
if and only if zero is contained in some generalized field of values derived from the
powers $I,A,\dots,A^k$. Most of the other previously published results on worst-case
GMRES are devoted to studying the tightness of the inequality
\begin{eqnarray}
\Psi_k(A) \;\leq\; \min_{p\in\pi_{k}}\|p(A)\|\;\equiv \varphi_{k}(A),\label{eqn:bound02}
\end{eqnarray}
which is easily derived from (\ref{eqn:WCapp}) using the submultiplicativity
property of the Euclidean norm. The quantity $\varphi_{k}(A)$ is called the
{\em $k$th ideal GMRES approximation}~\cite{GrTr1994}. The polynomial for
which the minimum is attained in (\ref{eqn:bound02}) is called the
{\em $k$th ideal GMRES polynomial} of $A$. This polynomial is uniquely
determined; see~\cite{GrTr1994,LiTi2009}. It was shown that (\ref{eqn:bound02})
is an equality for normal matrices $A$ and all $k\geq0$, and for $k=1$ and any
nonsingular $A$~\cite{GrGu1994,Jo1994}. Some nonnormal matrices~$A$
are known for which $\Psi_{k}(A)<\varphi_{k}(A)$, even $\Psi_{k}(A)\ll\varphi_{k}(A)$,
for certain~$k$; see~\cite{FaJoKnMa1996,To1997}.

The ideal GMRES approximation problem can be formulated as a semidefinite program
(see~\cite{ToTr1998}) and hence can be solved efficiently by standard software.
On the other hand, we are unaware of any efficient algorithm for solving
the worst-case GMRES approximation problem, so that in practice one needs
to resort to a ``general purpose'' nonlinear solver to compute worst-case
GMRES data. The difficult nonlinear nature of the worst-case GMRES approximation
problem may be one of the reasons why this problem is less studied
(both theoretically and numerically)
than the ideal GMRES approximation problem.

This paper is mainly devoted to characterizations of the worst-case GMRES problem (\ref{eqn:WCapp}).
We first show in Section~2 that $\Psi_{k}(A)=\Psi_{k}(A^T)$, and that worst-case initial
vectors satisfy a certain ``cross equality''. Next, in Section~3, we look at
the worst-case GMRES approximation problem from the optimization point of~view and show
that $k$th worst-case GMRES initial vectors are always right singular vectors of the
corresponding $k$th GMRES residual matrix. In Section~\ref{sec:Toh} we prove that a
$k$th worst-case GMRES polynomial may not be uniquely determined (unlike the $k$th
ideal GMRES polynomial), and we give a numerical example for two different
polynomials and corresponding initial vectors that both attain the same worst-case GMRES
value at the same step $k$. In Section~\ref{sec:WCvsID} we further study
differences between the worst-case and the ideal GMRES approximations. In particular,
we state a parameterized set of matrices $A$ of arbitrary size $2n$ (with $n\geq 2$)
for which the inequality in (\ref{eqn:bound02}) is sharp for all $k=3,\dots,2n-1$.
In the previously published examples in~\cite{FaJoKnMa1996,To1997}, a small matrix $A$
is constructed for which the sharp inequality occurs for exactly one $k$.
Finally, in Section~\ref{sec:rc} we analyze whether the values of the max-min approximation (\ref{eqn:WCapp})
and the min-max approximation (\ref{eqn:bound02}) for a real matrix change if we consider
the maximization over complex vectors and/or the minimization over complex polynomials.
This analysis gives another indication for the difference between the two approximation
problems.

\section{The cross equality}

In this section we generalize two results of Zavorin \cite{Za2001}. The first shows that
$\Psi_{k}(A)=\Psi_{k}(A^{T})$ and the second concerns a special property of worst-case
initial vectors (they satisfy the so-called ``cross equality''). Note that Zavorin proved
these results only for diagonalizable matrices using quite a complicated technique based
on the decomposition of the corresponding Krylov matrix.
Using a simple algebraic technique we prove these results for
general matrices. All results presented in this section
can easily be generalized from real to complex matrices.\medskip

\begin{theorem}\label{thm:01}
If $A\in\mathbb{R}^{n\times n}$ is a nonsingular matrix, then $\Psi_{k}(A)=\Psi_{k}(A^{T})$
for all $k=1,\dots,d(A)-1$.
\end{theorem}\medskip

\begin{proof}
Let $1\leq k\leq d(A)-1$ and
consider any unit norm vector $b$ such that the corresponding $k$th GMRES
residual vector $r_{k}=p_k(A)b$ is nonzero. The defining property (\ref{eqn:GMRESORTH})
of $r_k$ means that $\langle A^jb,r_k\rangle=0$ for $j=1,\dots,k$. Hence,
for any $q\in \pi_k$,
\begin{eqnarray}
\|r_k\|^2 = \langle p_k(A)b,r_{k}\rangle = \langle b,r_{k}\rangle =
\langle q(A)b,r_{k}\rangle = \langle b,q(A^{T})r_{k}\rangle
\leq  \|q(A^{T})r_{k}\|, \label{eqn:wc1}
\end{eqnarray}
where the last inequality follows from the Cauchy-Schwarz inequality and $\|b\|=1$.

If $b$ is a unit norm $k$th worst-case GMRES initial vector and $r_k$ is the corresponding
$k$th GMRES residual vector, then the previous inequality means that
\begin{equation}\label{eqn:wc2}
\|r_k\|^2 \;=\; \Psi_{k}^{2}(A)\;\leq\;\|q(A^{T})r_{k}\|,
\end{equation}
where $q\in\pi_k$ is arbitrary. Dividing by $\|r_k\|$ and taking the minimum over all $q\in\pi_{k}$
we get
\begin{equation}
\Psi_{k}(A)\leq\min_{q\in\pi_{k}}\left\Vert q(A^{T})\frac{r_{k}}{\|r_{k}\|}\right\Vert \leq\Psi_{k}(A^{T}).\label{eqn:wcez}
\end{equation}
Now we can reverse the roles of $A$ and $A^T$, and then repeat the whole argument
to obtain the opposite inequality, i.e., $\Psi_{k}(A^{T})\leq\Psi_{k}(A)$.
\end{proof}\medskip

The following theorem describes a special property of worst-case initial
vectors: If we apply GMRES to $A$ and a unit norm $k$th worst-case
initial vector $b$ giving at step $k$ the residual vector $r_{k}$, and
then $k$ steps of GMRES to $A^{T}$ and the initial vector $r_{k}/\|r_k\|$,
we obtain again the original initial vector $b$ (up to a scaling factor).\medskip

\begin{theorem}
\label{thm:cross}
Let $A\in\mathbb{R}^{n\times n}$ be a nonsingular matrix, and let $1\leq k\leq d(A)-1$.
If $b\in\mathbb{R}^{n}$ is a unit norm $k$th worst-case GMRES initial vector and
%\begin{eqnarray*}
\[
r_{k} \equiv  \mathrm{GMRES}(A,b,k),\qquad
s_{k}  \equiv  \mathrm{GMRES}\left(A^{T},\frac{r_{k}}{\|r_{k}\|},k\right),
\]
%\end{eqnarray*}
then
\[
\|s_{k}\|=\|r_{k}\|=\Psi_{k}(A)
\qquad\mbox{and}\qquad
b=\frac{s_{k}}{\Psi_{k}(A)}.
\]
\end{theorem}\medskip

\begin{proof}
Let $b$ be a unit norm $k$th worst-case GMRES initial vector and
let $r_k=\mathrm{GMRES}(A,b,k)$. In addition, let $s_k=\mathrm{GMRES}(A^T,r_k/\|r_k\|,k)$
and let $q_k$ be the corresponding $k$th GMRES polynomial. Using this polynomial
in (\ref{eqn:wc2}) yields
$$\|r_k\|=\Psi_k(A)\;\leq\;\left\Vert q_{k}(A^{T})\frac{r_{k}}{\|r_{k}\|}\right\Vert
\;=\;\|s_k\|\;\leq\; \Psi_k(A^T).$$
However, as shown in Theorem~\ref{thm:01}, equality holds throughout, which
shows the first assertion.

Moreover, since $\|r_k\|=\|s_k\|$, the (Cauchy-Schwarz) inequality on the right
of (\ref{eqn:wc1}) is an equality for the given $b$ and $q=q_k$, i.e.,
$$\langle b,q_k(A^T)r_k\rangle = \|q_k(A^T)r_k\|.$$
Since $\|b\|=1$, this happens if and only if
\[
b\;=\;\frac{q_{k}(A^{T})r_{k}}{\|q_{k}(A^{T})r_{k}\|}\;=\;
\frac{q_{k}(A^{T})r_{k}}{\|r_{k}\|\|r_{k}\|}\;=\;\frac{s_{k}}{\|r_{k}\|},
\]
which finishes the proof.
\end{proof}\medskip

The previous theorem shows that if $b$ is a unit norm
$k$th worst-case GMRES initial vector, then (with the same
notation as in the proof above)
$$\Psi_k(A)b\;=\;s_k\;=\;q_k(A^T)\frac{r_k}{\|r_k\|}=q_k(A^T)p_k(A)\frac{b}{\|r_k\|},$$
or, equivalently,
\begin{equation}\label{eqn:mmm}
q_{k}(A^{T})p_{k}(A)\, b\;=\;\Psi_{k}^{2}(A)\, b.
\end{equation}
In other words, $b$ is an eigenvector of the matrix
$q_{k}(A^{T})p_{k}(A)$
with the corresponding eigenvalue $\Psi_{k}^{2}(A)$. In Corollary~\ref{col:pAA} we will show
that $q_{k}=p_{k}$, i.e., that $b$ is a right singular vector of the $k$th
worst-case GMRES residual matrix $p_{k}(A)$.

To further investigate vectors with the special property introduced in
Theorem~\ref{thm:cross} we use the following definition.
\medskip

\begin{definition}\label{def:cross}
Let $A\in \mathbb{R}^{n\times n}$ be nonsingular. We say that a unit norm
vector $b\in\mathbb{R}^{n}$ with $d(A,b)>k$ satisfies the cross equality for $A$ and the
step $k\geq 1$, if
\[
b=\frac{s_{k}}{\|s_{k}\|},\quad\mbox{where}\quad
s_{k}\equiv \mathrm{GMRES}\left(A^{T},\frac{r_{k}}{\|r_{k}\|},k\right),\quad
r_{k}\equiv \mathrm{GMRES}(A,b,k).
\]
\end{definition}
%\medskip{}

\begin{algorithm}[ht]
\caption{(Cross iterations 1)}
\label{alg:cross}
\begin{algorithmic}[0]
\STATE $b^{(0)}=b$,
\FOR{$j=1,2,\dots$}
\STATE $r_{k}^{(j)}=\mathrm{GMRES}(A,b^{(j-1)},k)$
\STATE $c^{(j-1)}=r_{k}^{(j)}/\|r_{k}^{(j)}\|$
\STATE $s_{k}^{(j)}=\mathrm{GMRES}(A^{T},c^{(j-1)},k)$
\STATE $b^{(j)}=s_{k}^{(j)}/\|s_{k}^{(j)}\|$
\ENDFOR
\end{algorithmic}
\end{algorithm}

Inspired by Theorem~\ref{thm:cross} we define the iterative process shown
in Algorithm~\ref{alg:cross}.
To analyze this algorithm,
let us denote
$$r_k^{(j)}=p_k^{(j)}(A)b^{(j-1)}\quad\mbox{and}\quad
s_k^{(j)}=q_k^{(j)}(A^T)c^{(j-1)}.$$
Using $q=q_k^{(j)}$ in (\ref{eqn:wc1}) we then get
$$\|r_k^{(j)}\|^2\;\leq\; \|q_k^{(j)}(A^T)r_k^{(j)}\|\;=\;\|r_k^{(j)}\|\,
\|q_k^{(j)}(A^T)c^{(j-1)}\|\;=\;\|r_k^{(j)}\|\,\|s_k^{(j)}\|.$$
Now consider (\ref{eqn:wc1}) with the roles of $A$ and $A^T$ reversed, i.e.,
\begin{eqnarray*}
\|s_k^{(j)}\|^2 &=& \langle q_k^{(j)}(A^T)c^{(j-1)},s_k^{(j)}\rangle
\;=\;\langle c^{(j-1)}, q(A)s_k^{(j)}\rangle \;=\;
\|s_k^{(j)}\|\, \langle c^{(j-1)}, q(A)b^{(j)}\rangle\\
&\leq & \|s_k^{(j)}\|\, \|q(A)b^{(j)}\|,
\end{eqnarray*}
for all $q\in\pi_k$. We can choose $q=p_k^{(j+1)}$ and thus
obtain $\|s_k^{(j)}\|\leq \|r_k^{(j+1)}\|$. In summary, we have
shown that
\begin{equation}
\|r_{k}^{(j)}\|\leq\|s_{k}^{(j)}\|\leq\|r_{k}^{(j+1)}\|
\leq\|s_{k}^{(j+1)}\|\leq\Psi_{k}(A),\quad j=1,2,\dots\,.\label{eqn:sequence}
\end{equation}
Hence the sequences of norms $\|r_{k}^{(j)}\|$ and $\|s_{k}^{(j)}\|$,
$j=1,2,\dots\,$, interlace each other, are both nondecreasing, and are
both bounded by $\Psi_{k}(A)$. This implies that both sequences converge
to the same limit, which does not exceed $\Psi_{k}(A)$.

Consequently,
for any initial vector $b^{(0)}$, Algorithm~\ref{alg:cross} converges
to a vector that satisfies the cross equality for $A$ and step $k$.
If $b^{(0)}$ satisfies the cross equality for $A$ and step $k$,
then trivially equality holds in (\ref{eqn:sequence}) for all $j$.
On the other hand, if equality holds in (\ref{eqn:sequence}) for
one $j$, then, using \eqref{eqn:wc1},
$$
\langle b^{(j)},q_k^{(j)}(A^{T})r_{k}^{(j)}\rangle
=\|r_{k}^{(j)}\|^{2}=\|q_k^{(j)}(A^{T})r_{k}^{(j)}\|=\|s_k^{(j)}\|,
$$
and we have reached a vector that satisfies the cross equality.

From the above it is clear that the cross equality represents a necessary
condition for a vector $b^{(0)}$ to be a worst-case initial vector. On the
other hand, we can ask whether this condition is sufficient, or, at
least, whether the vectors that satisfy the cross equality are in
some sense special. To investigate this question we present the following
lemma.\medskip

\begin{lemma} Let $A\in\mathbb{R}^{n\times n}$ be nonsingular, $k\geq 1$,
 and $b \in\mathbb{R}^{n}$ be a unit norm initial vector with $d(A,b) > k$.
If $r_k = \mathrm{GMRES}(A,b,k)$, then $d(A^T, r_k) > k$, and $b$ satisfies the cross
for $A$ and the step~$k$ if and only if
$b\in\mathcal{K}_{k+1}(A^{T},r_k)$.
In particular, each unit norm vector $b$ with $d(A,b)=n$ satisfies the cross equality
for $A$ and the step $k=n-1$.
\end{lemma}

\medskip
\begin{proof}
The nonzero GMRES residual
$r_k\in b+A\mathcal{K}_{k}(A,b) \subset \mathcal{K}_{k+1}(A,b)$
is uniquely determined by the orthogonality conditions (\ref{eqn:GMRESORTH}),
which can be written as
\[
0=\langle A^{j}b,r_k\rangle=\langle b,(A^{T})^{j}r_{k}\rangle,\quad
\mbox{for}\quad j=1,\dots,k,
\]
or, equivalently,
\begin{equation}\label{eqn:bort}
b\perp A^{T}\mathcal{K}_{k}(A^{T},r_{k}).
\end{equation}
Now let $s_k\equiv \mbox{GMRES}(A^T,r_k/\|r_k\|,k)$.
From (\ref{eqn:wc1}) we know that $\|s_k\|\geq \|r_k\|>0$, i.e.
$d(A^T,r_k)>k$, and
\begin{equation}\label{eqn:sort}
s_{k}\in \frac{r_k}{\|r_k\|}+A^T\mathcal{K}_{k}(A^T,r_k)\subset
\mathcal{K}_{k+1}(A^T,r_k),\quad s_{k}\perp A^T\mathcal{K}_{k}(A^T,r_k)\,.
\end{equation}
If $b$ satisfies the cross equality for $A$ and the step $k$, then
$b=s_k/\|s_k\|$ and \eqref{eqn:sort} implies that $b\in\mathcal{K}_{k+1}(A^{T},r_k)$.
On the other hand, if $b\in\mathcal{K}_{k+1}(A^{T},r_k)$, then
$\langle b,r_k \rangle = \| r_k \|^2 \neq 0$ and \eqref{eqn:bort}
imply that $b=s_k/\|s_k\|$.

For $k=n-1$, we have $\mathcal{K}_{k+1}(A^T,r_k)=\mathbb{R}^n$,
i.e. $b \in \mathcal{K}_{k+1}(A^T,r_k)$ is always satisfied.
\end{proof}\medskip

\begin{figure}
\begin{center}
\includegraphics[width=6.25cm]{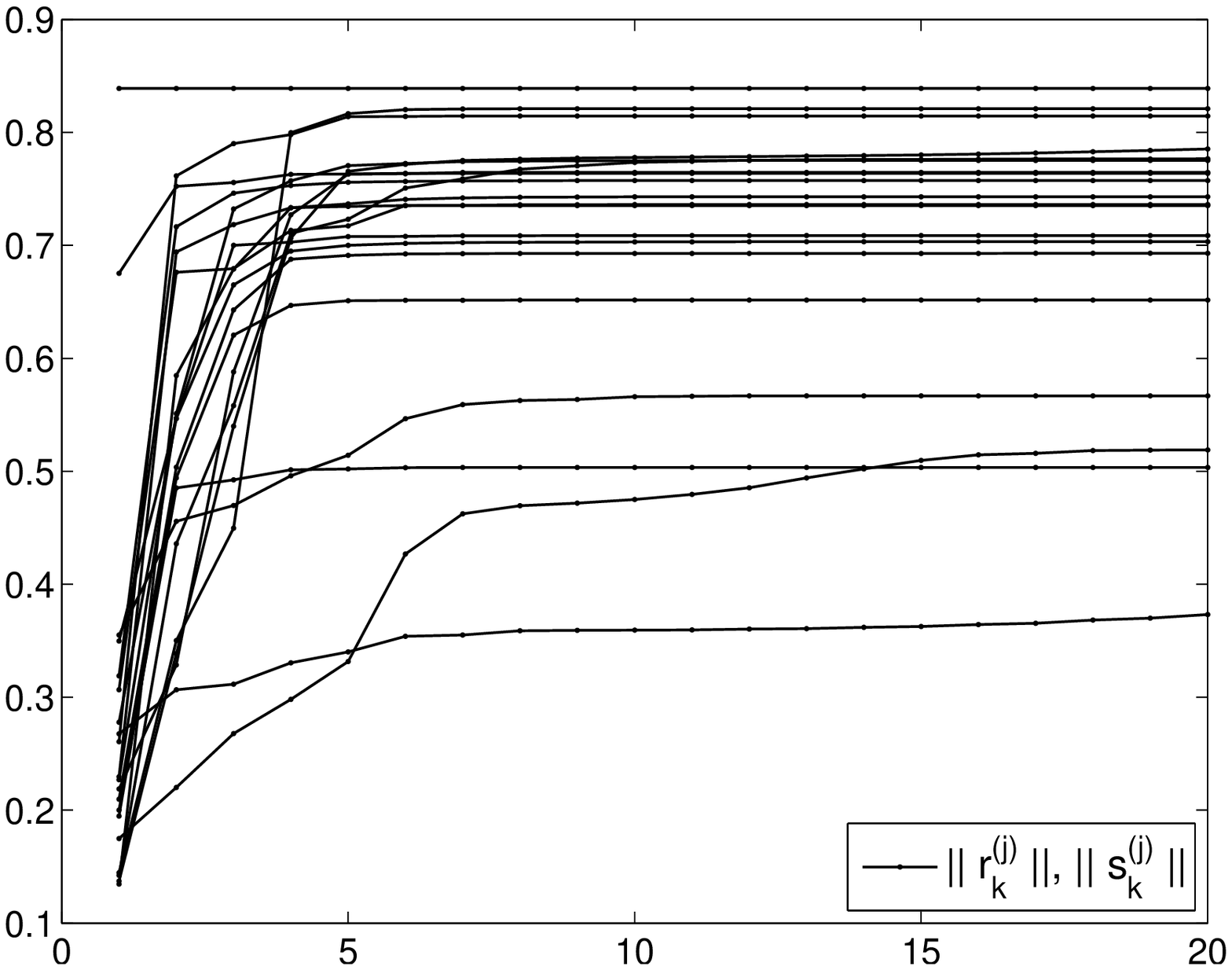}\hfill
\includegraphics[width=6.25cm]{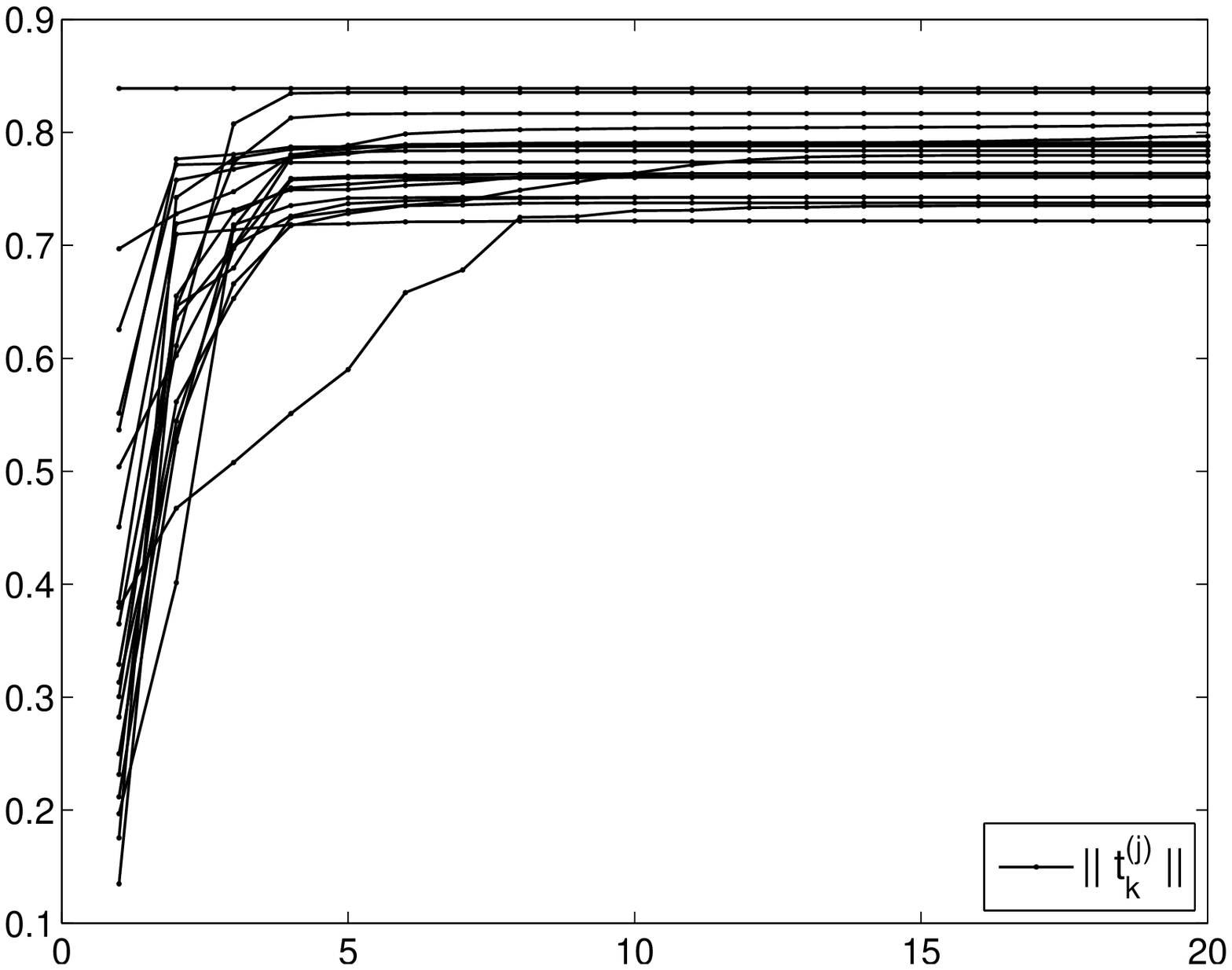}
\caption{Cross iterations for random initial vectors.}
\label{fig:cross}
\end{center}
\end{figure}

To give a numerical example for Algorithm~\ref{alg:cross} we consider $A$
being the Jordan block $J_\lambda$ of size $11$ with the eigenvalue
$\lambda=1$, and we and choose $k=5$. In this case, the ideal GMRES
matrix $\varphi_5(A)$ has a simple maximal singular value, as numerically
observed in \cite{TiLiFa2007}. Using the results of Greenbaum and
Gurvits in~\cite{GrGu1994} we know that then
$\Psi_5(J_\lambda)=\varphi_5(J_\lambda)$, and, moreover, that the
corresponding worst-case initial vector is the right singular vector
that corresponds to the maximal singular value of the ideal GMRES
matrix $\varphi_5(A)$. Hence, in this case the 5th worst-case initial
vector is uniquely determined up to scaling.

In the left part of \figurename~\ref{fig:cross} we show the results of
Algorithm~\ref{alg:cross} started with 20 random unit norm initial vectors.
Each line represents the sequence $\|r_{k}^{(j)}\|,\|s_{k}^{(j)}\|$,
for $j=1,\dots,10$. In the end of each of the 20 runs we get a vector
that satisfies (up to a small inaccuracy) the cross equality for $J_\lambda$
and $k=5$. We can observe that there are many initial vectors that satisfy the
cross equality, and there seems to be no special structure in the norms
that are attained in the end. In particular, none of the 20 runs
results in a 5th worst-case initial vector for which the norm $\Psi_5(A)$
is attained (this value is visualized by the highest horizontal line
in the figure).

We will now slightly modify the cross iteration Algorithm~\ref{alg:cross}.
Having a initial vector $b^{(j-1)}$  we always apply both, GMRES with
$A$ as well as GMRES with $A^T$, and look at the resulting GMRES residual
norm. We take as a resulting residual the one with the greater norm; see
Algorithm~\ref{alg:cross2}. After the process converges, we get again a vector
that satisfies the cross equality.

\begin{algorithm}[h]
\caption{(Cross iterations 2)}
\label{alg:cross2}
\begin{algorithmic}[0]
\STATE $b^{(0)}=b$,
\FOR{$j=1,2,\dots$}
\STATE $v =\mathrm{GMRES}(A,b^{(j-1)},k)$
\STATE $w =\mathrm{GMRES}(A^T,b^{(j-1)},k)$
 \IF{$\|v\|<\|w\|$}
	\STATE  $t_{k}^{(j)}=w$
 \ELSE
 	\STATE  $t_{k}^{(j)}=v$
 \ENDIF
\STATE $b^{(j)}=t_{k}^{(j)}/\|t_{k}^{(j)}\|$
\ENDFOR
\end{algorithmic}
\end{algorithm}

This strategy is a little better than the original one when looking for
a worst-case initial vector; see \figurename~\ref{fig:cross}.
While it is usually not sufficient
to find a worst-case vector, one at least can find a reasonable
initial point for an optimization procedure that solves the nonlinear
worst-case GMRES approximation problem.

\section{Optimization point of view}
Let a nonsingular matrix $A\in\mathbb{R}^{n\times n}$ and a positive
integer $k<d(A)$ be given. For vectors $c=[c_{1},\dots,c_{k}]^{T}\in\mathbb{R}^{k}$
and $v\in\mathbb{R}^{n}$, we define the function
\begin{equation}
f(c,v)\equiv\|p(A;c)v\|^{2}=v^{T}p(A;c)^{T}p(A;c)v,\label{eq:f1}
\end{equation}
where
\[
p(z;c)=1-\sum_{j=1}^{k}c_{j}z^{j}.
\]
Equivalently, we can express the function $f(c,v)$ using the matrix
\[
K(v)\equiv[Av,A^{2}v,\dots,A^{k}v]
\]
 as
\begin{eqnarray}
f(c,v) & = & \|v-K(v)c\|^{2}=v^{T}v-2v^{T}K(v)c+c^{T}K(v)^{T}K(v)c.\label{eq:funkcional}
\end{eqnarray}
(Here only the dependence on $v$ is expressed in the notation $K(v)$, because
$A$ and $k$ are both fixed.) Note that $K(v)^{T}K(v)$ is the Gramian matrix
of the vectors $Av,A^{2}v,\dots,A^{k}v$,
\[
K(v)^{T}K(v)=\left[v^{T}(A^{T})^{i}A^{j}v\right]_{i,j=1,\dots,k}.
\]
Next, we define the function
\[
g(v)\equiv\min_{c\in\mathbb{R}^{k}}f(c,v),
\]
which represents the $k$th squared GMRES residual norm for the matrix
$A$ and the initial vector $v$, and we denote
\[
\Omega\equiv\{u\in\mathbb{R}^{n}:\ d(A,u)\geq k\},\qquad\Gamma\equiv\{u\in\mathbb{R}^{n}:\ d(A,u)<k\}.
\]
The set $\Gamma$ is a closed subset, $\Omega$ is an open subset
of $\mathbb{R}^{n}$, and $\mathbb{R}^{n}=\Omega \cup \Gamma$. Note
that $g(v)>0$ for all $v\in\Omega$ and $g(v)=0$ for all $v\in\Gamma$.
The following lemma is a special case of \cite[Proposition 2.2]{FaJoKnMa1996}
for real data and nonsingular $A$.\medskip{}

\begin{lemma}
\label{lem:cinf} In the previous notation, the function $g(v)$ is
a continous function of $v\in\mathbb{R}^{n}$, i.e., $g\in C^{0}(\mathbb{R}^{n})$,
and it is an infinitely differentiable function of $v\in\Omega$,
i.e., $g\in C^{\infty}(\Omega)$. Moreover, $\Gamma$ has measure
zero in $\mathbb{R}^{n}$.
\end{lemma}
\medskip{}

We next characterize the minimizer of the function $f(c,v)$ as a
function of $v$.

\medskip{}

\begin{lemma}
For each given $v\in\Omega$, the problem
\[
\min_{c\in\mathbb{R}^{k}}f(c,v)
\]
 has the unique minimizer
\[
\gamma(v)=(K(v)^{T}K(v))^{-1}K(v)^{T}v\in\mathbb{R}^{k}.
\]
 As a function of $v\in\Omega$, this minimizer satisfies $\gamma(v)\in C^{\infty}(\Omega)$.
Given $v\in\Omega$, $(\gamma(v),v)$ is the only point in $\mathbb{R}^{k}\times\Omega$
with
\[
\nabla_c f(\gamma(v),v)=0.
\]
\end{lemma}

\begin{proof}
Since $v\in\Omega$ and $A$ is nonsingular, the vectors $Av,A^{2}v,\dots A^{k}v$
are linearly independent and $K(v)^{T}K(v)$ is symmetric and positive
definite. Therefore, if $v\in\Omega$ is fixed, \eqref{eq:funkcional}
is a quadratic functional in~$c$, which attains its unique global
minimum at the stationary point
\[
\gamma(v)=(K(v)^{T}K(v))^{-1}K(v)^{T}v.
\]
 The function $\gamma(v)$ is a well defined rational function of
$v\in\Omega$, and thus $\gamma(v)\in C^{\infty}(\Omega)$. Note that
the vector $\gamma(v)$ contains the coefficients of the $k$th GMRES
polynomial that corresponds to the initial vector $v\in\Omega$.
\end{proof}

\medskip{}

As stated in Lemma~\ref{lem:cinf}, $g(v)$
is a continuous function on $\mathbb{R}^n$, and thus it is also continuous
on the unit sphere
\[
S\equiv\{u\in\mathbb{R}^{n}:\ \|u\|=1\}.
\]
Since $S$ is a compact set and $g(v)$ is continuous on
this set, it attains its minimum and maximum on $S$.

We are interested in the characterization of points $(\tilde{c},\tilde{v})\in\mathbb{R}^{k}\times S$
such that
\begin{equation}
f(\tilde{c},\tilde{v})=\max_{v\in S}\min_{c\in\mathbb{R}^{k}}f(c,v)=\max_{v\in S}g(v).\label{eq:wcdef}
\end{equation}
 This is the worst-case GMRES problem (\ref{eqn:WCapp}).
Since $g(v)=0$ for all $v\in\Gamma$,
we have
\[
\max_{v\in S}g(v)=\max_{v\in S\cap\Omega}g(v).
\]
 To characterize the points $(\tilde{c},\tilde{v})\in\mathbb{R}^{k}\times S$
that satisfy (\ref{eq:wcdef}), we define for every $c\in{\mathbb{R}}^{k}$
and $v\neq0$ the two functions
\[
F(c,v)\equiv f\left(c,\frac{v}{\|v\|}\right)=\frac{f(c,v)}{v^{T}v},\qquad G(v)\equiv g\left(\frac{v}{\|v\|}\right)=\frac{g(v)}{v^{T}v}.
\]
 Clearly, for any $\alpha\neq0$, we have
\[
F(c,\alpha v)=F(c,v),\qquad G(\alpha v)=G(v).
\]

\begin{lemma}
It holds that $G(v)\in C^{\infty}(\Omega)$. A vector $\tilde{v}\in\Omega\cap S$
satisfies
\[
g(\tilde{v})\geq g(v)\quad\mbox{for all}\quad v\in S
\]
 if and only if $\tilde{v}\in\Omega\cap S$ satisfies
\[
G(\tilde{v})\geq G(v)\quad\mbox{for all}\quad v\in\mathbb{R}^{n}\backslash\{0\}.
\]
\end{lemma}

\begin{proof}
Since $g(v)\in C^{\infty}(\Omega)$ and $0\notin\Omega$, it holds
also $G(v)\in C^{\infty}(\Omega)$. If $\tilde{v}\in\Omega\cap S$
is a maximum of $G(v)$, then $\alpha\tilde{v}$ is a maximum as well,
so the equivalence is obvious.
\end{proof}
\medskip

\begin{theorem}
\label{thm:mainwc} The vectors $\tilde{c}\in\mathbb{R}^{k}$ and
$\tilde{v}\in S\cap\Omega$ that solve the problem
\[
\max_{v\in S}\min_{c\in\mathbb{R}^{n}}f(c,v)
\]
 satisfy
\begin{equation}
\nabla_c F(\tilde{c},\tilde{v})=0,\qquad
\nabla_v F (\tilde{c},\tilde{v})=0,\label{eq:stationary}
\end{equation}
 i.e., $(\tilde{c},\tilde{v})$ is a stationary point of the function
$F(c,v)$.
\end{theorem}

\medskip{}

\begin{proof}
Obviously, for any $v\in\Omega$,
\[
F(\gamma(v),v)=\frac{f(\gamma(v),v)}{v^{T}v}\leq\frac{f(c,v)}{v^{T}v}=F(c,v)
\quad\mbox{for all $c\in\mathbb{R}^{k}$,}
\]
 i.e., $\gamma(v)$ also minimizes the function $F(c,v)$ and
that
\[
\nabla_c F(\gamma(v),v)=0,\qquad v\in\Omega.
\]
 We know that $g(v)$ attains its maximum on $S$ at some point $\tilde{v}\in\Omega\cap S$.
Therefore, $G(v)$ attains its maximum also at $\tilde{v}$. Since
$G(v)\in C^{\infty}(\Omega)$, it has to hold that
\[
\nabla G (\tilde{v})=0.
\]
Denoting $\tilde{c}=\gamma(\tilde{v})$ and writing the function $G(v)$ as $G(v)=F(\gamma(v),v)$ we get
\begin{equation}\label{eqn:nabla}
\nabla G (\tilde{v})=0= \nabla_v \gamma(\tilde{v}) \nabla_c F(\tilde{c},\tilde{v})
+\nabla_v F(\tilde{c},\tilde{v}),
\end{equation}
where $\nabla_v \gamma(\tilde{v})$ is the $n\times k$ Jacobian matrix of the function $\gamma(v):\mathbb{R}^n \rightarrow \mathbb{R}^k$
at the point $\tilde{v}$. Here we used the standard chain rule for multivariate functions.
 Since $\tilde{v}\in\Omega\cap S$, we know from the previous that
$
\nabla_c F(\tilde{c},\tilde{v})=0,
$
and, therefore, using \eqref{eqn:nabla},
$
\nabla_v F(\tilde{c},\tilde{v})=0.
$
\end{proof}
\medskip{}

\begin{theorem}\label{thm:singular}
If $(\tilde{c},\tilde{v})$ is a solution of the problem \eqref{eq:wcdef},
then $\tilde{v}$ is a right singular vector of the matrix $p(A;\tilde{c})$.
\end{theorem}
\medskip{}

\begin{proof}
Since \emph{$(\tilde{c},\tilde{v})$} solves the problem \eqref{eq:wcdef},
we have
$
0=\nabla_v F(\tilde{c},\tilde{v}).
$
 Writing $F(c,v)$ as a Rayleigh quotient,
\[
F(c,v)=\frac{v^{T}p(A;c)^{T}p(A,c)v}{v^{T}v},
\]
 we ask when $\nabla_v F(c,v)=0$; for more details see \cite[pp. 114--115]{B:La2007}.
By differentiating $F(c,v)$ with respect to $v$ we get
\[
0=\frac{2p(A;c)^{T}p(A,c)v\:\|v\|^{2}-2v^{T}p(A;c)^{T}p(A,c)v\: v}{(v^{T}v)^{2}}
\]
 and the condition $0=\nabla_v F(\tilde{c},\tilde{v})$
is equivalent to
\[
p(A;\tilde{c})^{T}p(A,\tilde{c})\tilde{v}=F(\tilde{c},\tilde{v})\,\tilde{v}.
\]
 In other words, $\tilde{v}$ is a right singular vector of $p(A;\tilde{c})$
and $\sigma=\sqrt{F(\tilde{c},\tilde{v})}$ is the corresponding singular
value.
\end{proof}
\medskip{}

\begin{theorem}\label{thm:maximum}
A point $(\tilde{c},\tilde{v})\in\mathbb{R}^{k}\times S$ that solves
the problem \eqref{eq:wcdef} is a stationary point of $F(c,v)$ in
which the maximal value of $F(c,v)$ is attained.
\end{theorem}
\medskip{}

\begin{proof}
Using Theorem~\ref{thm:mainwc} we know that any solution $(\tilde{c},\tilde{v})\in\mathbb{R}^{k}\times S$
of (\ref{eq:wcdef}) is a stationary point of $F(c,v)$. On the other
hand, if $(\hat{c},\hat{v})\in\mathbb{R}^{k}\times S$ satisfies
\[
\nabla_v F (\hat{c},\hat{v})=0,\qquad\nabla_c F(\hat{c},\hat{v})=0,
\]
 then $p(A;\hat{c})$ is the GMRES polynomial that corresponds to
$\hat{v}$ and
\[
F(\hat{c},\hat{v})=\|p(A;\hat{c})\hat{v}\|^{2}\leq\|p(A;\tilde{c})\tilde{v}\|^{2}=F(\tilde{c},\tilde{v}).
\]
 Hence, $(\tilde{c},\tilde{v})$ is a stationary point of $F(c,v)$
in which the maximal value of $F(c,v)$ is attained.
\end{proof}
\medskip

As a consequence of previous results we can formulate the following corollary.
\medskip

\begin{corollary}
\label{col:pAA}
Let $A\in\mathbb{R}^{n\times n}$ be a nonsingular
matrix and let $1\leq k\leq d(A)-1$. Let $b$ be a $k$th unit norm
worst-case GMRES initial vector and let $p_k\in\pi_k$ be the
corresponding $k$th worst-case GMRES polynomial.
Then $p_k$ is also the $k$th worst-case GMRES polynomial for $A^T$ and the initial vector
$r_k/\|r_k\|$.
\end{corollary}

\medskip
\begin{proof} Using Theorem~\ref{thm:singular} and Theorem~\ref{thm:maximum} we know that
\begin{equation}\label{eqn:pp1}
	\Psi_k^2(A) b = p_k(A^T) p_k(A) b,
\end{equation}
i.e., that $b$ is a right singular vector of the GMRES residual matrix $p_k(A)$ that corresponds
to the maximal value of $F(\tilde{c},\tilde{v})$, i.e., to $\Psi_k^2(A)$. From \eqref{eqn:mmm} we also
know that
\begin{equation}\label{eqn:pp2}
\Psi_{k}^{2}(A)\, b=q_{k}(A^{T})p_{k}(A)\, b
\end{equation}
where $q_k$ is the GMRES polynomial that corresponds to $A^T$ and the initial vector $r_k$.
Comparing \eqref{eqn:pp1} and \eqref{eqn:pp2},
and using the uniqueness of GMRES polynomials it follows that $p_k=q_k$.
\end{proof}
\medskip

\section{Non-uniqueness of worst-case GMRES polynomials}\label{sec:Toh}

In this section we prove that a worst-case GMRES polynomial may not
be uniquely determined, and we give a numerical example for the occurrence
of a non-unique case. Our results are based on Toh's parameterized family of
(nonsingular) matrices
\begin{equation}
A=A(\omega,\varepsilon)=\left[\begin{array}{cccc}
1 & \varepsilon &  & \\
 & -1 & \frac{\omega}{\varepsilon} & \\
 &  & 1 & \varepsilon\\
 &  &  & -1
\end{array}\right]\in\mathbb{R}^{4\times 4},\qquad 0<\omega < 2,\quad0<\varepsilon.\label{eq:Toh}
\end{equation}
Toh used these matrices in~\cite{To1997} to show that $\Psi_3(A)/\varphi_3(A)\rightarrow 0$
for $\epsilon\rightarrow 0$ and each $\omega\in (0,2)$~\cite[Theorem~2.3]{To1997}. In other
words, he proved that the ratio of the worst-case and ideal GMRES approximations can be
arbitrarily small.

\medskip
\begin{theorem}\label{thm:Toh}
If $p_{k}(z)$ is a $k$th worst-case GMRES polynomial of $A$ in $\eqref{eq:Toh}$,
then $p_{k}(-z)$ is also a $k$th worst-case GMRES polynomial of $A$.

In particular, $p_3(z)\neq p_3(-z)$, so the third worst-case GMRES polynomial
of $A$ is not uniquely determined.
\end{theorem}
\medskip{}

\begin{proof}
Let $b$ be any unit norm $k$th worst-case initial vector of $A$, and consider
the orthogonal similarity transformation
\[
A=-QA^{T}Q^{T},\qquad Q=\left[\begin{array}{cccc}
 &  &  & 1\\
 &  & -1\\
 & 1\\
-1
\end{array}\right].
\]
Then
\[
p_{k}(A)b=Qp_{k}(-A^{T})Q^{T}b\qquad \mbox{and}\quad
\Psi_k(A)=\|p_{k}(A)b\|=\|p_{k}(-A^{T})w\|=\Psi_k(A^T),
\]
where $w=Q^{T}b$. In other words, $p_k(-z)$ is a $k$th worst-case GMRES polynomial
for $A^T$ and, using Corollary~\ref{col:pAA}, it is also a $k$th worst-case GMRES polynomial
for the matrix $A$.

Let $p_{3}(z)\in\pi_3$ be any third worst-case GMRES polynomial for the matrix $A$.
To show that $p_{3}(-z)\neq p_{3}(z)$ it suffices to show that $p_3(z)$
contains odd powers of $z$, i.e., that
\begin{equation}
p_3(z)\neq 1-\beta z^{2}\quad \mbox{for any $\beta\in\mathbb{R}$.}
\end{equation}
Define the matrix
\[
B\equiv\left[\begin{array}{cccc}
1 & 0 & \omega & 0\\
 & 1 & 0 & \omega\\
 &  & 1 & 0\\
 &  &  & 1
\end{array}\right]=A^{2}.
\]
From \cite[Theorem~2.1]{To1997} we know that the
(uniquely determined) third ideal GMRES polynomial of $A$ is of the form
\begin{equation}\label{eqn:Toh3}
p_{*}(z)=1+(\alpha-1)z^{2},\qquad\alpha=\frac{2\omega^{2}}{4+\omega^{2}}.
\end{equation}
Therefore,
\[
\min_{p\in\pi_{3}}\|p(A)\|=
\min_{p\in\pi_{1}}\max_{\|v\|=1}\|p(B)v\|=\max_{\|v\|=1}\min_{p\in\pi_{1}}\|p(B)v\|,
\]
where the last equality follows from the fact that the ideal and worst-case
GMRES approximations are equal for $k=1$~\cite{Jo1994,GrGu1994}. If a third
worst-case polynomial of $A$ is of the form $1-\beta z^{2}$ for some $\beta$, then
\[
\Psi_3(A)=\max_{\|v\|=1}\min_{p\in\pi_{3}}\|p(A)v\|=
\max_{\|v\|=1}\min_{p\in\pi_{1}}\|p(B)v\|=\min_{p\in\pi_{3}}\|p(A)\|=\varphi_3(A).
\]
This, however, contradicts the main result by Toh that
$\Psi_3(A)<\varphi_3(A)$; see~\cite[Theorem 2.2]{To1997}.
\end{proof}
\medskip

To compute examples of worst-case GMRES polynomials for the Toh matrix
$\eqref{eq:Toh}$ numerically we chose $\varepsilon=0.1$ and $\omega=1$,
and we used the function {\tt fminsearch} from Matlab's Optimization Toolbox.
We computed the value
$$
\Psi_3(A) = 0.4579
$$
(we present the numerical results only to 4 digits)
with the corresponding third worst-case initial vector
$$
	b = [-0.6376,
    0.0471,
   0.2188,
   0.7371]^T
$$
and the worst-case GMRES polynomial
\begin{eqnarray*}
p_{3}(z) & = & -0.025z^{3}-0.895z^{2}+0.243z+1=\frac{-1}{39.9}(z-1.181)(z+0.939)(z+35.96).
\end{eqnarray*}
One can numerically check that $b$ is the right singular vector
of $p_{3}(A)$ that corresponds to the second maximal singular
value of $p_{3}(A)$.
From Theorem~\ref{thm:Toh} we know that $q_3(z)\equiv p_3(-z)$ is also a third
worst-case GMRES polynomial. One can now find the corresponding worst-case
initial vector leading to the polynomial $q_3$ using the singular value
decomposition (SVD)
\[
p_{3}(A)=USV^{T},
\]
where the singular values are ordered nonincreasingly on the diagonal of $S$.
We know (by numerical observation) that $b$ is the second column of $V$.
We now compute the SVD of $q_3(A)$, and define the corresponding initial
vector as the right singular vector that corresponds to the second maximal
singular value of $q_3(A)$. It holds that
\[
p_{3}(A^T) = p_{3}(A)^T= V S U^{T}.
\]
Since $A^T = - Q A Q^T$, we get $Qp_{3}(-A)Q^T =  V S U^{T}$, or, equivalently,
\[
q_{3}(A) =  (Q^T V) S (Q^T U)^T.
\]
So, the columns of the matrix $Q^T U$ are right singular vectors of $q_{3}(A)$
and the vector  $Q^T u_2$, where $u_2$ is the second column of $U$,
is the worst-case initial vector that gives the worst-case GMRES
polynomial $q_3(z) = p_3(-z)$.

\section{Ideal versus worst-case GMRES phenomenon}\label{sec:WCvsID}

As mentioned above, Toh~\cite{To1997}
as well as Faber, Joubert, Knill, and Manteuffel~\cite{FaJoKnMa1996}
have shown that worst-case GMRES and ideal GMRES
are different approximation problems in the sense that there exist matrices
$A$ and iteration steps $k$ for which $\Psi_k(A)<\varphi_k(A)$. In this section
we further study these two approximation problems. We start with a geometrical
characterization related to the function $f(c,v)$ from (\ref{eq:funkcional}).
\medskip

\begin{theorem} Let $A\in\mathbb{R}^{n\times n}$ be a nonsingular
matrix and let $1\leq k\leq d(A)-1$.
The $k$th ideal and worst-case GMRES approximations are equal, i.e.,
\begin{equation}
\max_{v\in S}\min_{c\in\mathbb{R}^{k}}f(c,v)\
=\min_{c\in\mathbb{R}^{k}}\max_{v\in S}f(c,v),\label{eqn:wcideal}
\end{equation}
 if and only if $f(c,v)$ has a saddle point in $\mathbb{R}^{k}\times S$.
\end{theorem}
\medskip{}

\begin{proof}
If $f(c,v)$ has a saddle point in $\mathbb{R}^{k}\times S$, then
there exist vectors $\tilde{c}\in\mathbb{R}^{k}$ and $\tilde{v}\in S$
such that
\[
f(\tilde{c},v)\leq f(\tilde{c},\tilde{v})\leq f(c,\tilde{v})
\qquad\forall\, c\in\mathbb{R}^{k},\ \forall\, v\in S.
\]
 The condition $f(\tilde{c},v)\leq f(\tilde{c},\tilde{v})$ for all
$v\in S$ implies that $\tilde{v}$ is a maximal right singular vector
of the matrix $p(A;\tilde{c})$. If $f(\tilde{c},\tilde{v})\leq f(c,\tilde{v})$
for all $c\in\mathbb{R}^{k}$, then $p(z;\tilde{c})$ is the GMRES
polynomial that corresponds to the initial vector $\tilde{v}$. In
other words, if $f(c,v)$ has a saddle point in $\mathbb{R}^{k}\times S$,
then there exist a polynomial $p(z;\tilde{c})$ and a unit norm vector
$\tilde{v}$ such that $\tilde{v}$ is a maximal right singular vector
of $p(A;\tilde{c})$ and
\[
p(A;\tilde{c})\tilde{v}\perp A\mathcal{K}_{k}(A,\tilde{v}).
\]
 Using \cite[Lemma~2.4]{TiLiFa2007}, the $k$th ideal and worst-case
GMRES approximations are then equal.

On the other hand, if the condition (\ref{eqn:wcideal}) is satisfied,
then $f(c,v)$ has a saddle point in $\mathbb{R}^{k}\times S$.
\end{proof}
\medskip{}

In other words, the $k$th ideal and worst-case GMRES approximations are equal if
and only if the points $(\tilde{c},\tilde{v})\in\mathbb{R}^{k}\times S$
that solve the worst-case GMRES problem are also the saddle points
of $f(c,v)$ in $\mathbb{R}^{k}\times S$.

\medskip
We next extend the original construction of Toh~\cite{To1997} to obtain
some further numerical examples in which $\Psi_k(A)<\varphi_k(A)$. Note
that the Toh matrix $\eqref{eq:Toh}$ is not diagonalizable. In particular,
for $\omega=1$ we have $A=X \widetilde{J} X^{-1}$, where
$$
\widetilde{J} = \left[
         \begin{array}{cccc}
                 1 & 1 &  &    \\
                  &  1 &  &    \\
               &  & -1 & 1  \\
                &    &  & -1  \\
        \end{array}\right],\qquad
X = \left[
         \begin{array}{cccc}
                 \epsilon & \epsilon & \epsilon & -\epsilon   \\
                  -2 &  -1 & 0 & 1   \\
                   0 & -2\epsilon & 0 & 2\epsilon  \\
                 0 &  4  & 0& 0  \\
        \end{array}\right].
$$
One can ask whether the phenomenon $\Psi_k(A)<\varphi_k(A)$ can appear also
for diagonalizable matrices. The answer is yes, since both $\Psi_k(A)$ and $\varphi_k(A)$
are continuous functions on the open set of nonsingular matrices;
see~\cite[Theorem~2.5 and Theorem~2.6]{FaJoKnMa1996}. Hence one can
slightly perturb the diagonal of the Toh matrix $\eqref{eq:Toh}$
in order to obtain a diagonalizable
matrix $\widetilde{A}$ for which $\Psi_k(\widetilde{A})<\varphi_k(\widetilde{A})$.

For $\omega=1$, the Toh matrix
is an upper bidiagonal matrix with the alternating diagonal entries $1$ and $-1$,
and the alternating superdiagonal entries $\epsilon$ and $\epsilon^{-1}$.
One can consider such a matrix for any $n\geq 4$, i.e.,
$$
A = \left[\begin{array}{cccccc}
1 & \varepsilon\\
 & -1 & \varepsilon^{-1}\\
 &    & 1  & \varepsilon\\
 &  &  & \ddots&\ddots\\
&  &  & &  \ddots&\varepsilon^{\pm 1}\\
 &  &  & & & \pm 1\\
\end{array}\right]\in\mathbb{R}^{n\times n},
$$
and look at the values of $\Psi_k(A)$ and $\varphi_k(A)$. If $n$ is even, we found
numerically that $\Psi_k(A)=\varphi_k(A)$ for $k\neq n-1$ and $\Psi_{n-1}(A)<\varphi_{n-1}(A)$.
If $n$ is odd, then our numerical experiments showed that
$\Psi_k(A)=\varphi_k(A)$ for $k\neq n-2$ and $\Psi_{n-2}(A)<\varphi_{n-2}(A)$. Hence for all such matrices
worst-case and ideal GMRES differ from each other for exactly one $k$.

Inspired by the Toh matrix,
we define the $n\times n$ matrices (for any $n\geq 2$)
\[
J_{\lambda,\varepsilon}\equiv\left[\begin{array}{cccc}
\lambda & \varepsilon\\
 & \ddots & \ddots\\
 &  & \ddots & \varepsilon\\
 &  &  & \lambda
\end{array}\right],\qquad E_{\varepsilon}\equiv\left[\begin{array}{cccc}
0 & 0 & \dots & 0\\
\vdots &  &  & \vdots\\
0 & 0 & \dots & 0\\
\varepsilon^{-1} & 0 & \dots & 0
\end{array}\right]
\]
and use them to construct the matrix
\[
A=\left[\begin{array}{cc}
J_{1,\varepsilon} & \omega E_{\varepsilon}\\
 & J_{-1,\varepsilon}
\end{array}\right]\in\mathbb{R}^{2n\times 2n},\qquad \omega>0.
\]
One can numerically observe that here $\Psi_k(A)<\varphi_k(A)$
for all steps $k=3,\dots, 2n-1$. As an example, we plot in
\figurename~\ref{fig:cross-2} the ideal and worst-case GMRES convergence
curves for $n=4$, i.e., $A$ is an $8\times8$
matrix, $\omega=4$ and $\varepsilon=0.1$.
Varying the parameter~$\omega$ will influence the difference between
worst-case and ideal GMRES in these examples.

\begin{figure}
\begin{center}
\includegraphics[width=6.25cm]{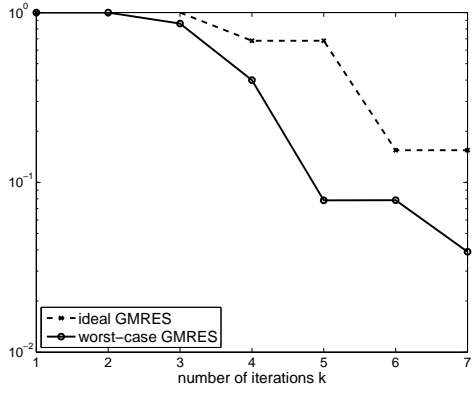}
\end{center}
\caption{Ideal and worst-case GMRES can differ from step 3 up to the step $2n-1$.}
\label{fig:cross-2}
\end{figure}

\medskip
\section{Ideal and worst-case GMRES for complex vectors or polynomials}\label{sec:rc}
We now ask whether the values of the max-min approximation (\ref{eqn:WCapp}) and the
min-max approximation (\ref{eqn:bound02}) for a matrix $A\in\mathbb{R}^{n\times n}$
can change if we allow the maximization over complex vectors and/or the minimization
over complex polynomials. The answer to this question will show that the two approximation
problems indeed are of a different nature.

Let us define
\[
\varphi_{k,\mathbb{K},\mathbb{F}}(A) \equiv\min_{p\in\pi_{k,\mathbb{K}}}
\max_{{b\in\mathbb{F}^{n}\atop \|b\|=1}}\|p(A)b\|,\qquad
\Psi_{k,\mathbb{K},\mathbb{F}}(A) \equiv\max_{{b\in\mathbb{F}^{n}\atop \|b\|=1}}
\min_{p\in\pi_{k,\mathbb{K}}}
\|p(A)b\|,
\]
where $\mathbb{K}$ and $\mathbb{F}$ are either the real or the complex numbers. Hence,
the previously used  $\varphi_k(A)$, $\Psi_k(A)$, and $\pi_k$ are now denoted by
$\varphi_{k,\mathbb{R},\mathbb{R}}(A)$ and $\Psi_{k,\mathbb{R},\mathbb{R}}(A)$,
and $\pi_{k,\mathbb{R}}$, respectively. We first analyze the case of
$\varphi_{k,\mathbb{K},\mathbb{F}}(A)$.\medskip

\begin{theorem}\label{thm:ideal} For a nonsingular matrix $A\in\mathbb{R}^{n\times n}$ and
$1\leq k\leq d(A)-1$,
\[
\varphi_{k,\mathbb{R},\mathbb{R}}(A)=\varphi_{k,\mathbb{C},\mathbb{R}}(A)=
\varphi_{k,\mathbb{R},\mathbb{C}}(A)=\varphi_{k,\mathbb{C},\mathbb{C}}(A).
\]
\end{theorem}

\begin{proof} Since
\[
\max_{{b\in\mathbb{R}^{n}\atop \|b\|=1}}\|Bv\|=\|B\|=\max_{{b\in\mathbb{C}^{n}\atop \|b\|=1}}\|Bv\|
\]
holds for any real matrix $B\in\mathbb{R}^{n\times n}$, we have
$
\varphi_{k,\mathbb{R},\mathbb{R}}(A)=\varphi_{k,\mathbb{R},\mathbb{C}}(A).
$

Next, from $\mathbb{R}\subset\mathbb{C}$ we get immediately
$
\varphi_{k,\mathbb{C},\mathbb{R}}(A)\leq\varphi_{k,\mathbb{R},\mathbb{R}}(A).
$
On the other hand, writing $p\in\pi_{k,\mathbb{C}}$ in the form
$p=p_{r}+\mathbf{i}\: p_{i}$,
where $p_{r}\in\pi_{k,\mathbb{R}}$ and $p_{i}$ is a real polynomial
of degree at most $k$ such that $p_i(0)=0$, we get
\begin{eqnarray*}
\varphi_{k,\mathbb{C},\mathbb{R}}^{2}(A)=\min_{p\in\pi_{k,\mathbb{C}}}\max_{{b\in\mathbb{R}^{n}\atop \|b\|=1}}\|p(A)b\|^{2} & = & \min_{p\in\pi_{k,\mathbb{C}}}\max_{{b\in\mathbb{R}^{n}\atop \|b\|=1}}\,\left(\|p_{r}(A)b\|^{2}+\|p_{i}(A)b\|^{2}\right)\\
 & \geq & \min_{p_{r}\in\pi_{k,\mathbb{R}}}\max_{{b\in\mathbb{R}^{n}\atop \|b\|=1}}\|p_{r}(A)b\|^{2}=\varphi_{k,\mathbb{R},\mathbb{R}}^{2}(A),
\end{eqnarray*}
so that $\varphi_{k,\mathbb{C},\mathbb{R}}(A)=\varphi_{k,\mathbb{R},\mathbb{R}}(A).$
Finally, from~\cite[Theorem~3.1]{Jo1994a}
we obtain $\varphi_{k,\mathbb{R},\mathbb{R}}(A)=\varphi_{k,\mathbb{C},\mathbb{C}}(A).$
\end{proof}\medskip

Since the value of $\varphi_{k,\mathbb{K},\mathbb{F}}(A)$ does not change
when choosing for $\mathbb{K}$ and $\mathbb{F}$ real or complex numbers,
we will again use the simple notation  $\varphi_{k}(A)$ in the following text.
The situation for the quantities corresponding to the worst-case GMRES approximation
is more complicated. Our proof of this fact uses the following lemma.\medskip

\begin{lemma}\label{lem:Toh}
If $A=A(\omega,\varepsilon)$ is the Toh matrix defined in $(\ref{eq:Toh})$ and
\begin{equation}\label{eqn:B}
B\equiv\left[\begin{array}{cc}
A & 0\\
0 & A
\end{array}\right],
\end{equation}
then
$\Psi_{3,\mathbb{R},\mathbb{R}}(B)=
\varphi_{3}(A)$.
\end{lemma}
\begin{proof}
Using the structure of $B$ it is easy to see that $\Psi_{k,\mathbb{R},\mathbb{R}}(B)\leq \varphi_{k}(A)$
for any~$k$.
To prove the equality, it suffices to find a real unit norm vector
$w$ with
\begin{equation}\label{eqn:BA}
\min_{p\in\pi_{3,\mathbb{R}}}\|p(B)w\|=\varphi_{3}(A)=\min_{p\in\pi_{3,\mathbb{R}}}\|p(A)\|.
\end{equation}
The solution $p_*$ of the ideal GMRES problem on the right hand side of (\ref{eqn:BA}) is
given by (\ref{eqn:Toh3}).
Toh  showed in~\cite[p.~32]{To1997} that $p_*(A)$ has a
twofold maximal singular value $\sigma$, and that the corresponding
right and left singular vectors are given (up to a normalization) by
\[
[v_{1},v_{2}]=\left[\begin{array}{rr}
0 & c\\
c & 0\\
0 & -2\\
-2 & 0
\end{array}\right],\qquad [u_{1},u_{2}]=\left[\begin{array}{rr}
0 & 2\\
2 & 0\\
0 & -c\\
-c & 0
\end{array}\right],
\]
i.e.,
$\sigma u_{1}=p_{*}(A)v_{1}$ and $\sigma u_{2}=p_{*}(A)v_{2},$
where $\sigma=\|p_{*}(A)\|$.

Let us define
\[
w\equiv \left[\begin{array}{c}
v_{1}\\
v_{2}
\end{array}\right]/\left\Vert \left[\begin{array}{c}
v_{1}\\
v_{2}
\end{array}\right]\right\Vert ,\qquad q(z)\equiv p_{*}(z).
\]
Using
\[
q(B)\left[\begin{array}{c}
v_{1}\\
v_{2}
\end{array}\right]=\sigma\left[\begin{array}{c}
u_{1}\\
u_{2}
\end{array}\right]\quad\mbox{and}\quad
\left\Vert \left[\begin{array}{c}
v_{1}\\
v_{2}
\end{array}\right]\right\Vert =\left\Vert \left[\begin{array}{c}
u_{1}\\
u_{2}
\end{array}\right]\right\Vert,
\]
we see that $\left\Vert q(B)w\right\Vert =\sigma.$ To prove (\ref{eqn:BA})
it is sufficient to show that $q$ is the third GMRES polynomial for
$B$ and $w$, i.e., that $q$ satisfies
$q(B)w\perp B^{j}w$ for $j=1,2,3$,
or, equivalently,
\[
\left[\begin{array}{c}
u_{1}\\
u_{2}
\end{array}\right]^{T}\left[\begin{array}{cc}
A^{j} & 0\\
0 & A^{j}
\end{array}\right]\left[\begin{array}{c}
v_{1}\\
v_{2}
\end{array}\right]=u_{1}^{T}A^{j}v_{1}+u_{2}^{T}A^{j}v_{2}=0,
\quad j=1,2,3.
\]
Using linear algebra calculations we get
$u_{1}^{T}Av_{1}=-4c=-u_{2}^{T}Av_{2}$,
and
\[
0=u_{1}^{T}A^{2}v_{1}=u_{2}^{T}A^{2}v_{2}=u_{1}^{T}A^{3}v_{1}=u_{2}^{T}A^{3}v_{2}.
\]
Therefore, we have found a unit norm initial vector $w$ and the corresponding
third GMRES polynomial $q$ such that $\|q(B)w\|=\varphi_{3}(A).$\qquad \end{proof}\medskip

We next analyze the quantities $\Psi_{k,\mathbb{K},\mathbb{F}}(A)$.\medskip

\begin{theorem}\label{thm:wc}
For a nonsingular matrix $A\in\mathbb{R}^{n\times n}$ and $1\leq k\leq d(A)-1$,
\[
\Psi_{k,\mathbb{R},\mathbb{R}}(A)=\Psi_{k,\mathbb{C},\mathbb{R}}(A)\leq \Psi_{k,\mathbb{C},\mathbb{C}}(A)
\leq \Psi_{k,\mathbb{R},\mathbb{C}}(A) \,,
\]
where both inequalities can be strict.
\end{theorem}\medskip
\begin{proof}
For a real initial vector $b$, the corresponding GMRES polynomial is uniquely determined and real.
This implies
$
\Psi_{k,\mathbb{C},\mathbb{R}}(A)=\Psi_{k,\mathbb{R},\mathbb{R}}(A).
$
Next, from~\cite[Theorem~3.1]{Jo1994a} it follows that
$
\Psi_{k,\mathbb{R},\mathbb{R}}(A)\leq\Psi_{k,\mathbb{C},\mathbb{C}}(A).
$
Finally, using $\mathbb{R}\subset\mathbb{C}$ we get
$
\Psi_{k,\mathbb{C},\mathbb{C}}(A)\leq\Psi_{k,\mathbb{R},\mathbb{C}}(A).
$

It remains to show that the inequalities can be strict. For the first inequality,
as shown in~\cite[Section~4]{ZaOlEl2003},
there exist real matrices $A$ and certain complex (unit norm) initial
vectors $b$ for which $\min_{p\in\pi_{k,\mathbb{C}}} \,\|p(A)b\|=1$ for
$k=1,\dots,n-1$ (complete stagnation), while such complete stagnation does
not occur for any real (unit norm) initial vector. Therefore,
there are matrices for which $\Psi_{k,\mathbb{C},\mathbb{R}}(A)< \Psi_{k,\mathbb{C},\mathbb{C}}(A)$.

{To show that the second inequality can be strict,
we note that for any $A\in\mathbb{R}^{n\times n}$, the corresponding
matrix $B\in\mathbb{R}^{2n\times 2n}$
of the form (\ref{eqn:B}), and $1\leq k\leq d(A)-1$,
\begin{eqnarray}
\Psi^2_{k,\mathbb{R},\mathbb{C}}(A)&=&\max_{{b\in\mathbb{C}^{n}\atop \|b\|=1}}\min_{p\in\pi_{k,\mathbb{R}}}\|p(A)b\|^2
=\max_{{u,v\in\mathbb{R}^{n}\atop \|u\|^2+\|v\|^2=1}}\min_{p\in\pi_{k,\mathbb{R}}}\|p(A)(u+\,\mathbf{i}\,v)\|^2\nonumber\\
&=&\max_{{u,v\in\mathbb{R}^{n}\atop \|u\|^2+\|v\|^2=1}}\min_{p\in\pi_{k,\mathbb{R}}}
\,(\|p(A)u\|^2+\|p(A)v\|^2)\nonumber\\
&=&\max_{{v\in\mathbb{R}^{2n}\atop \|v\|=1}}\min_{p\in\pi_{k,\mathbb{R}}}\|p(B)v\|^2\;=\;
\Psi^2_{k,\mathbb{R},\mathbb{R}}(B).\label{eqn:ident}
\end{eqnarray}
Now let $A$ be the Toh matrix $(\ref{eq:Toh})$ and $k=3$. Toh showed in \cite[Theorem~2.2]{To1997}
that for any unit norm $b\in\mathbb{C}^{4}$ and the corresponding third GMRES polynomial
$p_{b}\in\pi_{3,\mathbb{C}}$,
\[
\|p_{b}(A)b\|<\varphi_{3}(A).
\]
Hence $\Psi_{3,\mathbb{C},\mathbb{C}}(A)<\varphi_{3}(A)$. Lemma~\ref{lem:Toh} and
equation (\ref{eqn:ident}) imply
$\varphi_{3}(A)=\Psi_{3,\mathbb{R},\mathbb{C}}(A)$, which completes the proof of the
strict inequality.\qquad}
\end{proof}

\medskip
Our proof concerning the strictness of the first inequality in the previous theorem relied on
a numerical example given in~\cite[Section~4]{ZaOlEl2003}. We will now give an alternative construction
based on the non-uniqueness of the worst-case GMRES polynomial, which will lead to an example with
\[
\Psi_{k,\mathbb{R},\mathbb{R}}(A)<\Psi_{k,\mathbb{R},\mathbb{C}}(A).
\]
Suppose that $A$ is a real matrix for which in a certain step $k$ two {\em different} worst-case
polynomials $p_b\in\pi_{k,\mathbb{R}}$ and $p_c\in\pi_{k,\mathbb{R}}$ with corresponding real
unit norm initial vectors $b$ and $c$ exist, so that
$$\Psi_{k,\mathbb{R},\mathbb{R}}(A)\,=\,\|p_{b}(A)b\|\,=\,\|p_c(A)c\|.$$
Note that since $p_b$ and $p_c$ are the uniquely determined GMRES polynomials that solve the
problem (\ref{eqn:GMRESAPP}) for the corresponding real initial vectors, it holds that
\begin{equation}\label{eqn:strict}
\|p_{b}(A)b\| < \|p(A)b\|,\qquad \|p_{c}(A)c\| < \|p(A)c\|
\end{equation}
for any polynomial $p\in\pi_{k,\mathbb{C}}\setminus \{p_b,p_c\}$.

Writing any complex vector $w\in\mathbb{C}^{n}$ in the form
$w=(\cos\theta)\,u\,+\,\mathbf{i}\,(\sin\theta)\,v$, with $u,v\in\mathbb{R}^{n}$,
$\|u\|=\|v\|=1,$
we get
\begin{eqnarray*}
\Psi^2_{k,\mathbb{R},\mathbb{C}}(A) &=&
\max_{w\in\mathbb{C}^{n}\atop \|w\|=1}\min_{p\in\pi_{k,\mathbb{R}}}\,\|p(A)b\|^2 \\
&=& \max_{\theta\in\mathbb{R},u,v\in\mathbb{R}^{n}\atop \|u\|=\|v\|=1}
\min_{p\in\pi_{k,\mathbb{R}}}\,
\left(\cos^2\theta\,\|p(A)u\|^2+\sin^2\theta\,\|p(A)v\|^2\right)\\
&\geq &
\max_{\theta\in\mathbb{R}}
\min_{p\in\pi_{k,\mathbb{R}}}\,
\left(\cos^2\theta\|p(A)b\|^2+\sin^2\theta\|p(A)c\|^2\right)\\ &>&
(\cos^2\theta)\,\Psi^2_{k,\mathbb{R},\mathbb{R}}(A)\,+\,(\sin^2\theta)\,\Psi^2_{k,\mathbb{R},\mathbb{R}}(A)
=\Psi^2_{k,\mathbb{R},\mathbb{R}}(A),
\end{eqnarray*}
where the strict inequality follows from (\ref{eqn:strict}) and from the fact that $\|p(A)b\|^2$
and $\|p(A)c\|^2$ do not attain their minima for the same polynomial.

To demonstrate the strict inequality
$\Psi_{k,\mathbb{R},\mathbb{R}}(A)<\Psi_{k,\mathbb{R},\mathbb{C}}(A)$
numerically we use the Toh matrix $\eqref{eq:Toh}$ with
$\varepsilon=0.1$ and $\omega=1$, and $k=3$.
Let $b$ and $c$ be the corresponding two
different worst-case initial vectors introduced in Section~\ref{sec:Toh}.
We vary $\theta$ from $0$ to $\pi$ and compute the quantities
\begin{equation}\label{eqn:wccompl}
\min_{p\in\pi_{3,\mathbb{R}}}\,
\left(\cos^2\theta\,\|p(A)b\|^2+\sin^2\theta\,\|p(A)c\|^2\right)
=
\min_{p\in\pi_{3,\mathbb{R}}}\,\|p(B)g_\theta\|^2,
\end{equation}
where
\[
B=\left[\begin{array}{cc}
A & 0\\
0 & A
\end{array}\right]\quad\mbox{and}\quad
g_\theta=\left[\begin{array}{c} (\cos \theta)b\\
(\sin \theta)c
\end{array}\right].
\]
In \figurename~\ref{fig:wcomplex}
we can see clearly, that for $\theta\notin\{0,\,\pi/2\,,\pi\}$ the value
of (\ref{eqn:wccompl}) is strictly larger than $\Psi_3(A)=0.4579$.

\begin{figure}
\begin{center}
\includegraphics[width=6.25cm]{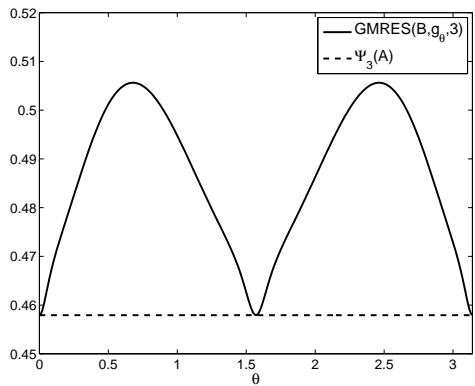}
\caption{The GMRES residual norm for a varying complex right hand side.}
\label{fig:wcomplex}
\end{center}
\end{figure}

\section{Concluding remarks}
We have studied the worst-case GMRES approximation problem, which for each (nonsingular) matrix $A$
and iteration step $k\leq d(A)$ represents the best possible attainable upper bound on the actual
GMRES residual norm for a linear algebraic system with $A$ at step $k$. We have derived several
theoretical properties of the worst-case GMRES problem, and we have studied its relation to the
ideal GMRES approximation problem.

In this paper we did not consider quantitative estimation of the worst-case GMRES value $\Psi_k(A)$,
and we did not study how this value depends on properties of $A$. This is an important problem of
great practical interest, which is largely open. For more details and a survey of the current
state-of-the-art we refer to~\cite[Section~5.7]{LieStrBook12}.

\bibliographystyle{siam}
\bibliography{paper}

\end{document}